\documentclass[12pt]{article}
\usepackage{amssymb,amsfonts,amsthm}
\usepackage{amsmath}
\usepackage{amstext}
\usepackage{mathrsfs}
\usepackage{mathtools}
\usepackage{tensor}
\usepackage{color}
\usepackage[latin1]{inputenc}
\usepackage[active]{srcltx}
\def\openC{{\rm C\kern-.18cm\vrule width.8pt height 7pt depth-.2pt \kern.18cm}}
\def\openN{{{\rm I}\kern-.16em {\rm N}}}
\def\openR{{{\rm I}\kern-.16em {\rm R}}}
\def\openT{{{\rm T}\kern-.42em {\rm T}}}
\def\openZ{{{\rm Z}\kern-.28em{\rm Z}}}

\newtheorem{thm}{Theorem}[section]

\newtheorem{lem}[thm]{Lemma}
\newtheorem{prop}[thm]{Proposition}

\theoremstyle{definition}

\begin{document}

\title{
{\textbf{{Strict positive definiteness on a product of compact two-point homogeneous spaces}}} \vspace{-4pt}
\author{\sc V. S. Barbosa and V. A. Menegatto}
}
\date{}
\maketitle \vspace{-30pt}
\bigskip

\begin{center}
\parbox{13 cm}{\small We present an explicit characterization for the real, continuous, isotropic and strictly positive definite kernels on a product of
compact two-point homogeneous spaces, in the cases in which at least one of the spaces is a sphere of dimension greater than 1 and the other is not a
circle.\ The result complements similar characterizations previously obtained for products of high dimensional spheres.   }
\end{center}


\medskip


\thispagestyle{empty}

\section{Introduction}

Let $\mathbb{M}^d$ denote a $d$-dimensional compact two-point homogeneous space.\ A real and continuous kernel $K$ on $\mathbb{M}^d$ is {\em positive definite} if it is a symmetric function, that is,
$$K(x,y)=K(y,x), \quad x,y \in \mathbb{M}^d,$$ and satisfies the following condition: if $n$ is a positive integer and $x_1,x_2, \ldots, x_n$ are distinct points on $\mathbb{M}^d$, then the $n \times n$ matrix
$[K(x_i,x_j)]$ is nonnegative definite.\ This formulation for the concept of positive definiteness is a refinement of the classical one described in \cite{berg}.\ A positive definite kernel on $\mathbb{M}^d$ is {\em strictly positive definite} if all the matrices in the previous definition are in fact positive definite.\ Positive definite kernels and functions on metric spaces are an important subject in many areas of classical and modern mathematics, such as radial basis function interpolation and approximation, geomathematics, geostatistics, Fourier analysis, etc.\ The interested reader may consider the references \cite{cheney, freeden, gneiting, wendland} as a starting point for ratifying that.

The continuity of a positive definite kernel on $\mathbb{M}^d$ is attached to the usual (geodesic) distance on $\mathbb{M}^d$, here denoted by $|xy|$, $x, y \in \mathbb{M}^d$.\ Throughout the paper, we will assume the geodesic distance on $\mathbb{M}^d$ fulfills the following requirement: all geodesics have the same length $2\pi$.\ The geodesic distance on $\mathbb{M}^d$ allows the introduction of the notion of {\em isotropy} of a positive definite kernel $K$ on $\mathbb{M}^d$, in the same way I. J. Schoenberg did for positive definite kernels on spheres (\cite{schoenberg}).\ It demands that
$$
K(x,y)=K_i^d(\cos |xy|/2), \quad x,y \in \mathbb{M}^d,
$$
for some function $K_i^d: [-1,1] \to \mathbb{R}$, here called the {\em isotropic part} of $K$.\ According to R. Gangolli (\cite{gangolli}), a continuous and isotropic kernel $K$ on $\mathbb{M}^d$ is positive definite if and only if
\begin{equation}\label{PDM}
K_i^d(t)=\sum_{k=0}^{\infty}a_k^{(d-2)/2,\beta} P_k^{(d-2)/2,\beta}(t), \quad t \in [-1,1],
\end{equation}
in which $a_k^{(d-2)/2,\beta} \in [0,\infty)$, $k\in \mathbb{Z}_+$ and $\sum_{k=0}^{\infty}a_k^{(d-2)/2,\beta} P_k^{(d-2)/2,\beta}(1) <\infty$.\ Here, $\beta=(d-2)/2, -1/2, 0, 1, 3$, depending on the respective category $\mathbb{M}^d$ belongs to, among the following ones (\cite{wang}): the unit spheres $S^d$, $d=1,2,\ldots$, the real projective spaces $\mathbb{P}^d(\mathbb{R})$, $d=2,3,\ldots$, the complex projective spaces
$\mathbb{P}^d(\mathbb{C})$, $d=4,6,\ldots$, the quaternionic projective spaces $\mathbb{P}^d(\mathbb{H})$, $d=8,12,\ldots$, and the Cayley projective plane $\mathbb{P}^{d}(Cay)$, $d=16$.\ The symbol $P_k^{(d-2)/2,\beta}$ stands for the Jacobi polynomial of degree $k$ associated with the pair $((d-2)/2,\beta)$.

This is the point where we can explain what the intentions in this paper are.\ The first target is to present a characterization for the real, continuous and isotropic kernels which are positive definite on a cartesian product of compact two-point homogeneous spaces.\ If $\mathbb{M}^d$ and $\mathbb{H}^{d'}$ are the spaces, the isotropy of a kernel $K: \mathbb{M}^d \times \mathbb{H}^{d'} \to \mathbb{R}$ corresponds to isotropy in both spaces, that is, $K$ is of the form
$$K((x,w),(y,z))=K_i^{d,d'}(\cos (|x y|/2), \cos (|wz|/2)), \quad x,y \in \mathbb{M}^d,\quad w,z \in \mathbb{H}^{d'},$$
for some function $K_i^{d,d'}: [-1,1]^2 \to \mathbb{R}$ (the isotropic part of $K$).\ Since we believe such a characterization can be deduced via classical results from harmonic analysis (see \cite{bachoc}), the proposal here is to obtain the characterization using an alternative procedure involving the Gauss hypergeometric function $\tensor[_2]{F}{_1}$ and basic convergence arguments.\ The characterization itself can be described as follows: if $K$ is a real, continuous and isotropic kernel on $\mathbb{M}^d \times \mathbb{H}^{d'}$, then it is positive definite if and only if its isotropic part has a series representation in the form
$$
K_i^{d,d'}(t,s) = \sum_{k,l=0}^\infty a_{k,l}(K_i^{d,d'})P_k^{(d-2)/2,\beta}(t)P_l^{(d'-2)/2,\beta'}(s), \quad t,s \in [-1,1]^2,
$$
in which $a_{k,l}(K_i^{d,d'})\geq 0$, $k,l\in \mathbb{Z}_+$ and $\sum_{k,l=0}^\infty a_{k,l}(K_i^{d,d'})P_k^{(d-2)/2,\beta}(1)P_l^{(d'-2)/2,\beta'}(1)<\infty$.\ The numbers $\beta$ and $\beta'$ have to agree, respectively, with
$d$ and $d'$ respecting Wang's classification in \cite{wang}.\ The result will be deduced in Section 2.

As explained in \cite{guella1}, positive definiteness on a product of spaces allows intermediate notions of strict positive definiteness.\ In the case of a positive definite kernel $K: \mathbb{M}^d \times \mathbb{H}^{d'} \to \mathbb{R}$, one of them reads like this: $K$ is $DC$-{\em strictly positive definite} if the strict positive definiteness condition previously introduced holds for points of $\mathbb{M}^d \times \mathbb{H}^{d'}$ having distinct components.\ Equivalently, all matrices of the form $[K((x_i,w_i),(x_j,w_j))]$ are positive definite if the $x_i$ are distinct in $\mathbb{M}^d$ and the $w_i$ are distinct in $\mathbb{H}^{d'}$.\ Clearly, $DC$-strict positive definiteness is a weaker notion in the sense that it demands plain strict positive definiteness for just some of the distinct points in $\mathbb{M}^d \times \mathbb{H}^{d'}$.\ In the case $\mathbb{M}^d =S^d$ and $\mathbb{H}^{d'}=S^{d'}$, $d,d'\geq 2$, $DC$-strict positive definiteness of a real, continuous, isotropic and positive definite kernel $K$ was shown to be equivalent to the following condition (\cite{guella3}): the set $\{k+l: a_{k,l}(K_i^{d,d'})>0\}$ contains infinitely many even and infinitely many odd integers.\ In Section 3, we will complete this line of investigation, presenting a characterization for $DC$-strict positive definiteness in all the other products involving compact two-point homogeneous spaces, except when one of the spaces is a circle.

The third target in the paper is to present a characterization for plain strict positive definiteness of a real, continuous, isotropic and strict positive definite kernel on $\mathbb{M}^d \times \mathbb{H}^{d'}$, in the same cases mentioned at the end of the previous paragraph.\ That is a counterpart of similar results for products of spheres previously obtained in \cite{guella2,guella3,guella4}.\ The details will appear in Section 4.

\section{Positive definiteness}

In this section, we will provide a characterization for the real, continuous, isotropic and positive definite kernels on $\mathbb{M}^d \times \mathbb{H}^{d'}$.\ No additional assumption on either $d$ or $d'$ will be made.\ The characterization is an extension to all the compact two-point homogeneous spaces of that one previously obtained in \cite{guella1} in the case both spaces are spheres (see also \cite{berg} for an alternative proof in that case).

Let us begin with some basics on Jacobi polynomials.\ For  $\alpha,\beta>-1$, the set $\{ P_k^{\alpha,\beta}: k\in \mathbb{Z}_+\}$ of Jacobi polynomials associated to the pair $(\alpha,\beta)$ is orthogonal on $[-1,1]$ in the sense that
$$
\int_{-1}^1 P_k^{\alpha,\beta}(t)P_l^{\alpha,\beta}(t)(1-t)^{\alpha}(1+t)^{\beta}dt = \delta_{k,l} h_k^{\alpha,\beta},
$$
where
$$
h_k^{\alpha,\beta} = \frac{2^{\alpha+\beta+1}}{2k+\alpha+\beta+1}\frac{\Gamma(k+\alpha+1)\Gamma(k+\beta+1)}{\Gamma(k+1)\Gamma(k+\alpha+\beta+1)},\quad k\in \mathbb{Z}_+.
$$
Here, and in many other places in the paper, $\Gamma$ will stand for the usual gamma function.\ An immediate consequence is the orthogonality of the family
$$
\left\{ (t,s)\in [-1,1]\mapsto P_k^{\alpha,\beta}(t)P_l^{\alpha',\beta'}(s):k,l\in \mathbb{Z}_+\right\}
$$
with respect to the weight function
$$
\sigma_{d}^{d'}(t,s):=(1-t)^\alpha(1+t)^\beta(1-s)^{\alpha'}(1+s)^{\beta'}, \quad  t,s \in [-1,1].
$$

There exists a generating formula for Jacobi polynomials via the {\em Gauss hypergeometric function} $\tensor[_2]{F}{_1}$ (\cite{askey,erdelyi,rainville,szego}).\ As a regular solution of the hypergeometric differential equation,
the hypergeometric function has a representation in the form
$$
\tensor[_2]{F}{_1}(a,b;c;z) = \sum_{n=0}^\infty \frac{(a)_n(b)_n z^n}{(c)_n n!},
$$
in which $a,b$ and $c$ are generic parameters, $z$ is a complex variable, and
$$
(\lambda)_n=\frac{\Gamma(n+\lambda)}{\Gamma(\lambda)} =  \left\{ \begin{array}{ccc} \lambda(\lambda+1)\cdots(\lambda+n-1), & \hbox{ if } & n\geq 1\\
1, & \hbox{ if } & n=0
\end{array}\right.
$$
is the Pochhammer symbol.\ The convergence holds for $|z|<1$ if $c$ is not a negative integer and for $|z|=1$ if $\mbox{Re\,}(c-a-b)>0$ (\cite[p. 63]{szego}).

For simplicity's sake, we will write
$$F(a,b;c;z) := \tensor[_2]{F}{_1}(a,b;c;z).$$
The hypergeometric function $F$ is differentiable with respect to $z$ in $\{z\in \mathbb{C}: |z|<1\}$ (\cite[p. 281]{whittaker}) and
$$
\frac{d}{dz}F(a,b;c;z)=\frac{ab}{c}F(a+1,b+1;c+1;z).
$$
In particular,
$$
\int_{z_1}^{z_2}F(a,b;c;z)dz=\frac{c-1}{(a-1)(b-1)}\left[F(a-1,b-1;c-1;z_2)-F(a-1,b-1;c-1;z_1) \right].
$$
A generating formula for Jacobi polynomials based on the function $F$ is the content of the following Poisson formula (\cite[p. 21]{askey}):
$$
\sum_{n=0}^\infty \frac{P_n^{\alpha,\beta}(1)P_n^{\alpha,\beta}(t)r^n}{h_n^{\alpha,\beta}} = G^{\alpha,\beta}(r) F\left(\frac{\alpha+\beta+2}{2},\frac{\alpha+\beta+3}{2};\beta+1;\frac{2r(1+t)}{(1+r)^2}\right), \quad t \in [-1,1],
$$
in which
$$
G^{\alpha,\beta}(r):=\frac{2^{-(\alpha+\beta+1)}\Gamma(\alpha+\beta+2)(1-r)}{\Gamma(\alpha+1)\Gamma(\beta+1)(1+r)^{\alpha+\beta+2}}.
$$
We now detach a consequence of the results described above to be used ahead.

\begin{lem}\label{poisson2} Let $\alpha\geq \beta\geq-1/2$ and $r\in(-1,1)$.\ The series
$$
\sum_{n=0}^\infty \frac{P_n^{\alpha,\beta}(1)P_n^{\alpha,\beta}(t)r^n}{h_n^{\alpha,\beta}}
$$
is convergent for all $t\in [-1,1]$.
\end{lem}
\begin{proof}If $r=0$, the result is obvious.\ Otherwise, taking into account the convergence of the series representation for $F$, it is promptly seen that the series in the statement of the lemma will be convergent as long as
$$
|1+t|<\frac{(1+r)^2}{2|r|}.
$$
However, a simple calculation reveals that $(1+r)^2> 4|r|$ whenever $r\in(-1,1)\setminus\{0\}$.\ Thus, since $t\in[-1,1]$, the convergence follows in this case as well.
\end{proof}

Lemma \ref{isopart} is a critical step towards the desired characterization in this section.\ It follows from the definition of positive definiteness along with Gangolli's characterization for positive definiteness on a single compact two-point homogeneous space and the Schur product theorem for nonnegative definite matrices.\ Once again, $\beta$ and $\beta'$ have to agree with
Wang's classification for the compact two-point homogeneous spaces.

\begin{lem}\label{isopart} For fixed nonnegative integers $k$ and $l$, the function
$$
(t,s) \in [-1,1]\mapsto P_k^{(d-2)/2,\beta}(t)P_l^{(d'-2)/2,\beta'}(s)
$$
is the isotropic part of a positive definite kernel on $\mathbb{M}^d\times \mathbb{H}^{d'}$.
\end{lem}

In the next proposition, $dx$ and $dy$ will denote the volume elements on $\mathbb{M}^d$ and $\mathbb{H}^{d'}$, respectively (dimensions will be omitted).\ We will require the expansion of a function $f$ from
$L_1([-1,1],\sigma_d^{d'})$:
$$
f(t,s)\sim \sum_{k,l=0}^\infty a_{k,l}(f)P_k^{(d-2)/2,\beta}(t)P_l^{(d'-2)/2,\beta'}(s)
$$
in which
$$
 a_{k,l}(f)=\frac{1}{h_k^{(d-2)/2,\beta} h_l^{(d'-2)/2,\beta'}}\int_{[-1,1]^2} f(t,s) P_k^{(d-2)/2,\beta}(t)P_l^{(d'-2)/2,\beta'}(s)d\sigma_d^{d'}(t,s), \quad k,l\in \mathbb{Z}_+.
$$
The dependence of $a_{k,l}$ upon $d,d',\beta,\beta'$ will be omitted.

In the case $f$ is the isotropic part of a kernel belonging to the setting of the paper, the following formula holds.

\begin{prop}\label{FH2}  Let $k$ and $l$ be nonnegative integers.\ If $K$ is a real, continuous and isotropic kernel on
$\mathbb{M}^d\times \mathbb{H}^{d'}$, then there exists a positive constant $C$, that depends upon $d,d',\beta$ and $\beta'$, so that
\begin{eqnarray*}
a_{k,l}(K_i^{d,d'}) & = & C \int_{\mathbb{M}^d\times \mathbb{H}^{d'}} \left[ \int_{\mathbb{M}^d\times \mathbb{H}^{d'}} K((x,y),(w,z))P_k^{(d-2)/2,\beta}(\cos (|x y|/2))\right.\\
 & & \hspace*{35mm} \left.\times  P_l^{(d'-2)/2,\beta'}(\cos (|wz|/2)) dydz\right]dxdw.
\end{eqnarray*}
\end{prop}
\begin{proof} The first step in the proof is to observe that the integral
$$
\int_{\mathbb{M}^d\times \mathbb{H}^{d'}} f(\cos (|x y|/2), \cos (|wz|/2))  P_k^{(d-2)/2,\beta}(\cos (|x y|/2)P_l^{(d-2)/2,\beta'}(\cos (|wz|/2))dy dz
$$
equals to a positive multiple of $a_{k,l}(f)$.\ Indeed, this follows after two applications of the Funk-Hecke formula for harmonics on compact two-point homogeneous spaces (see Proposition 2.8 and the Remark 2.9 in \cite{meaney})
coupled with Fubini's theorem.\ The multiple depends upon $d,d',\beta$ and $\beta'$.\ Calling it $C$ and integrating leads to the formula in the statement of the proposition.
\end{proof}

If we introduce the positive definiteness of $K$ as an assumption, we obtain the following expected consequence.

\begin{prop}\label{positive} If $K$ is a real, continuous, isotropic and positive definite kernel on $\mathbb{M}^d\times \mathbb{H}^{d'}$, then
$$
 a_{k,l}(K_i^{d,d'})\geq 0, \quad k,l \in \mathbb{Z}_+.
$$
\end{prop}
\begin{proof} Due to Lemma \ref{isopart}, if $K$ is positive definite on $\mathbb{M}^d\times \mathbb{H}^{d'}$, then the integrand in the double integral appearing in the statement of the previous proposition defines a real, continuous, isotropic and positive definite on $\mathbb{M}^d\times \mathbb{H}^{d'}$.\ Hence, the the proof of the proposition resumes to showing that if $K$ is a real, continuous, isotropic and positive definite kernel on
$\mathbb{M}^d\times \mathbb{H}^{d'}$, then the integral
$$I:=\int_{\mathbb{M}^d\times \mathbb{H}^{d'}} \left[ \int_{\mathbb{M}^d\times \mathbb{H}^{d'}} K((x,y),(w,z)) dydz\right]dxdw$$
is nonnegative.\ In order to achieve that, we will fix $\epsilon >0$ and will show there exists a number $\overline{I}=\overline{I}(\epsilon) \geq 0$ so that
$$|I-\overline{I}|\leq \epsilon \mbox{Vol.}(\mathbb{M}^d\times \mathbb{H}^{d'})^2.$$
Indeed, if $I$ were negative, then the information above with a convenient choice for $\epsilon$ would produce a contradiction.\ Since $\mathbb{M}^d\times \mathbb{H}^{d'}$ is a compact metric space, the kernel $K$ is actually uniformly continuous on $\mathbb{M}^d\times \mathbb{H}^{d'}$.\ In particular, we can select $\delta>0$ so that
$|K((x,y),(w,z))-K((x',y'),(w',z'))|<\epsilon$ whenever $x,x',y,y'\in \mathbb{M}^d$, $w,w',z,z' \in \mathbb{H}^{d'}$, $|xx'|<\delta$, $|yy'|<\delta$, $|ww'|<\delta$ and $|zz'|<\delta$.\ Since the metric spaces
$\mathbb{M}^d$ and $\mathbb{H}^{d'}$ are totally bounded, we can cover them with finitely many open balls of radius $\delta/2$.\ Likewise, we can cover $\mathbb{M}^d\times \mathbb{H}^{d'}$ with finitely many open balls.\ Using the covering, we can partition $\mathbb{M}^d\times \mathbb{H}^{d'}$ into finitely many Borel subsets
$$\mathbb{M}^d\times \mathbb{H}^{d'}= \sqcup_{j=1}^p B_j$$
so that $|xx'|<\delta$ and $|ww'|<\delta$ whenever $((x,w),(x',w'))\in B_j$, $j=1,2,\ldots,p$.\ It is now clear that if, $((x,y),(w,z)),((x',y'), (w',z'))\in B_j\times B_k$, for some pair $(k,j)$, then $|K((x,y),(w,z))-K((x',y'),(w',z'))|<\epsilon$.\ To proceed, observe that
$$I=\sum_{l,k=1}^p \int_{B_j}\int_{B_k} K((x,y),(w,z))dxdwdydz.$$
Next, for each $j \in \{1,2,\ldots,p\}$, choose $(x_j,w_j)\in B_j$ and define $\lambda_j:=$ the volume of $B_j$.\ Clearly, the number
$$\overline{I}:=\sum_{j,k=1}^p \lambda_j \lambda_k K((x_j,x_k),(w_j,w_k))$$
is nonnegative due to the positive definiteness of the kernel $K$.\ On the other hand
\begin{eqnarray*}
|I-\overline{I}| & = & \left|\sum_{j,k=1}^p\int_{B_j}\int_{B_k}  [K((x,y),(w,z))- K((x_j,x_k),(w_j,w_k))]dxdwdydz\right|\\
                 & \leq & \epsilon \sum_{j,k=1}^p \lambda_j \lambda_k,
\end{eqnarray*}
that is,
$|I-\overline{I}| \leq \epsilon\, \mbox{Vol.}(\mathbb{M}^d\times \mathbb{H}^{d'})^2$.
\end{proof}

Next, we move to convergence of double series defined by Jacobi polynomials.

\begin{lem} \label{conve1} Let $K$ be a real, continuous, isotropic and positive definite kernel on $\mathbb{M}^d\times \mathbb{H}^{d'}$.\ If $r,\rho \in (-1,1)$, then the double series
$$
\sum_{k,l=0}^\infty a_{k,l}(K_i^{d,d'})P_k^{(d-2)/2,\beta}(1)P_l^{(d'-2)/2,\beta'}(1)r^k\rho^l
$$
converges.
\end{lem}
\begin{proof} We will prove the lemma in the case in which $d\neq 1$ and $d'\neq 1$.\ The proof in the cases in which either $d=1$ or $d'=1$ can be adapted from the general proof presented below.\ However, the computation
referring to the the case in which the space is $S^1$ needs to be done directly without mentioning the hypergeometric function.\ The calculations made in \cite{guella1} are very similar to what is needed in these specials cases.\ 
Otherwise, the general term of the series in the statement of the lemma is
$$
\int_{-1}^1\int_{-1}^1 K_i^{d,d'}(t,s)\frac{P_k^{(d-2)/2,\beta}(1)}{h_k^{(d-2)/2,\beta}} P_k^{(d-2)/2,\beta}(t)r^k \frac{P_l^{(d'-2)/2,\beta'}(1)}{h_l^{(d'-2)/2,\beta'}}P_l^{(d'-2)/2,\beta'}(s)\rho^l d\sigma_d^{d'}(t,s).
$$
Introducing the Gauss hypergeometric function in the above expression leaves the double series in the form
$$\hspace*{-15mm}
\int_{-1}^1\int_{-1}^1 \left[ K_i^{d,d'}(t,s) G^{(d-2)/2,\beta}(r) F\left(\frac{d+2\beta+2}{4},\frac{d+2\beta+4}{4};\beta+1;\frac{2r(1+t)}{(1+r)^2}\right)\right.$$
 $$\hspace*{15mm} \times\left. G^{(d'-2)/2,\beta'}(\rho) F\left(\frac{d'+2\beta'+2}{4},\frac{d'+2\beta'+4}{4};\beta'+1;\frac{2\rho(1+s)}{(1+\rho)^2}\right) d\sigma_d^{d'}(t,s) \right].$$
Due to the positive definiteness of $K$ and the continuity of $K_i^{d,d'}$ in $[-1,1]\times [-1,1]$, we can estimate the double integral above by
$$\hspace*{-30mm}C^{r,\rho} \int_{-1}^1\int_{-1}^1 \left[  F\left(\frac{d+2\beta+2}{4},\frac{d+2\beta+4}{4};\beta+1;\frac{2r(1+t)}{(1+r)^2}\right)\right.$$
$$ \hspace*{30mm}\times \left. F\left(\frac{d'+2\beta'+2}{4},\frac{d'+2\beta'+4}{4};\beta'+1;\frac{2\rho(1+s)}{(1+\rho)^2}\right) d\sigma_d^{d'}(t,s) \right].$$
in which $C^{r,\rho}$ is a positive multiple of $G^{(d-2)/2,\beta}(r) G^{(d'-2)/2,\beta'}(\rho)$.\ The weight in the definition of $d\sigma_d^{d'}$ can be bounded by $2^{\beta+\beta'-2+(d+d')/2}$.\ Introducing this bound and solving the resulting integrals, we conclude that the double series is at most
$$ C\,\, F\left(\frac{d+2\beta-2}{4},\frac{d+2\beta}{4};\beta;\frac{4r}{(1+r)^2}\right)\\
  F\left(\frac{d'+2\beta'-2}{4},\frac{d'+2\beta'}{4};\beta';\frac{4\rho}{(1+\rho)^2}\right),
$$
in which
$$C=G^{(d-2)/2,\beta}(r) G^{(d'-2)/2,\beta'}(\rho)\frac{\beta \beta' 2^{\beta+\beta'+4+(d+d')/2}}{(d+2\beta-2)(d'+2\beta'-2)(d+\beta)(d'+\beta)}\frac{(1+r)^2(1+\rho)^2}{r\rho}.$$
The proof is complete.
\end{proof}

\begin{prop} \label{conve} If $K$ is a real, continuous, isotropic and positive definite kernel on $\mathbb{M}^d\times \mathbb{H}^{d'}$.\ then the double series
$$
\sum_{k,l=0}^\infty a_{k,l}(K_i^{d,d'})P_k^{(d-2)/2,\beta}(t)P_l^{(d'-2)/2,\beta'}(s)
$$
converges absolutely and uniformly for $(t,s)\in [-1,1]^2$.
\end{prop}
\begin{proof}
Due to the Weierstrass $M$-test for double series, it suffices to show that
$$
\sum_{k,l=0}^\infty a_{k,l}(K_i^{d,d'})P_k^{(d-2)/2,\beta}(1)P_l^{(d'-2)/2,\beta'}(1)
$$
converges.\ In order to do that, consider the sequence $(s_{p,q})_{p,q\in \mathbb{Z_+}}$ given by the partial sums
$$
s_{p,q}:=\sum_{k=0}^p\sum_{l=0}^q a_{k,l}(K_i^{d,d'})P_k^{(d-2)/2,\beta}(1)P_l^{(d'-2)/2,\beta'}(1),\quad p,q\in \mathbb{Z}_+.
$$
By Lemma \ref{positive}, $a_{k,l}(K_i^{d,d'})\geq 0$, for all $k,l\in \mathbb{Z}_+$.\ In particular, $s_{p,q}\leq s_{p',q'}$ when $p\leq p'$ and $q\leq q'$.\ On the other hand, by the previous lemma,
$$
\sum_{k=0}^p\sum_{l=0}^q a_{k,l}(K_i^{d,d'})P_k^{(d-2)/2,\beta}(1)P_l^{(d'-2)/2,\beta'}(1)r^k\rho^l\leq C,\quad p,q\in \mathbb{Z}_+, \quad r,\rho\in(-1,1).
$$
for some $C>0$.\ Applying the limits when $r,\rho\to 1^+$, we deduce the sequence $(s_{p,q})$ is bounded above.\ The convergence of $(s_{p,q})$ follows.
\end{proof}

The main result in the section is as follows.

\begin{thm}\label{gangolliext} Let $K$ be a real, continuous and isotropic kernel on $\mathbb{M}^d\times \mathbb{H}^{d'}$.\ It is positive definite on $\mathbb{M}^d\times \mathbb{H}^{d'}$ if and only if its isotropic part $K_i^{d,d'}$ has a representation in the form
$$
K_i^{d,d'}(t,s) = \sum_{k,l=0}^\infty a_{k,l}(K_i^{d,d'})P_k^{(d-2)/2,\beta}(t)P_l^{(d'-2)/2,\beta'}(s), \quad t,s \in [-1,1]^2,
$$
in which $a_{k,l}(K_i^{d,d'})\geq 0$, $k,l\in \mathbb{Z}_+$ and $\sum_{k,l=0}^\infty a_{k,l}(K_i^{d,d'})P_k^{(d-2)/2,\beta}(1)P_l^{(d'-2)/2,\beta'}(1)<\infty$.
\end{thm}
\begin{proof}
Consider the function $g$ defined by the Fourier expansion
$$g(t,s) \sim \sum_{k,l=0}^\infty a_{k,l}(K_i^{d,d'})P_k^{(d-2)/2,\beta}(t)P_l^{(d'-2)/2,\beta'}(s),\quad t,s \in [-1,1].$$
If $K$ is positive definite, then Proposition \ref{conve1} guarantees the convergence of the series for $t=s=1$.\ Proposition \ref{conve} implies convergence for all the other values of $t$ and $s$ while Proposition \ref{positive} yields that all the coefficients in the expansion are nonnegative.\ Since $g$ is continuous and the Fourier coefficients of $K_i^{d,d'}$ coincide with those of $g$, it follows that $K_i^{d,d'}=g$.\ This takes care of one implication in the theorem.\ As for the other, it follows from Lemma \ref{isopart} and the fact that the pointwise limit of positive definite kernels is itself positive definite.
\end{proof}

\section{$DC$-strict positive definiteness}

Either one of the concepts of strict positive definiteness we have introduced so far, demands considering $n\times n$ matrices $A=[A_{\mu\nu}]$ with
$$A_{\mu\nu}=K_i^{d,d'}(\cos{(|x_\mu x_\nu|/2)},\cos{(|w_\mu w_\nu|/2)}),$$
in which $K_i^{d,d'}$ is the isotropic part of the kernel and $(x_\mu,w_\mu)$, $\mu=1,2,\ldots,n$, are distinct points in $\mathbb{M}^d \times \mathbb{H}^{d'}$.\ Analyzing the associated quadratic forms
$$c^tAc:=\sum_{\mu,\nu=1}^n c_\mu c_\nu K((x_\mu,w_\mu),(x_\nu,w_\nu)), \quad c_\mu \in\mathbb{R}, \quad \mu=1,2,\ldots,n,$$
it is possible to obtain a quite more convenient formulation for either concept.\ We will proceed discussing $DC$-strict positive definiteness and will just mention the formulation for plain strict positive definiteness later.\
From now on, if $K$ is a real, continuous, isotropic and positive definite kernel $K$ on $\mathbb{M}^d\times \mathbb{H}^{d'}$, we will use the following notation attached to the series representation of its isotropic part:
$$J_K:=\left\{(k,l): a_{k,l}^{(d-2)/2,\beta}>0\right\}.$$
At this point, we need the addition formula demonstrated by Giné (\cite{gine,koo}), that is,
$$
\sum_{j=1}^{\delta(k,d)} S_{k,j}^d(x)\overline{S_{k,j}^d(y)} = c_k^{d,\beta} P_k^{(d-2)/2,\beta}\left(\cos{(|xy|/2)}\right),\quad x,y\in \mathbb{M}^d,
$$
where
$$
c_{k}^{d,\beta} := \frac{\Gamma(\beta+1)(2k+(d-2)/2+\beta+1)\Gamma(k+(d-2)/2+\beta+1)}{\Gamma((d-2)/2+\beta+2)\Gamma(k+\beta+1)}.
$$
The set $\{S^d_{k,1},S^d_{k,2},\hdots,S^d_{k,\delta(k,d)}\}$ denotes an orthonormal basis of the space $\mathcal{H}^d_k$ of spherical harmonics of degree $k$ on $\mathbb{M}^d$.

If we consider the representation for $K$ provided by Theorem \ref{gangolliext} and the addition formula above, then the equality $c^tAc=0$ corresponds to
$$
\sum_{k,l=0}^\infty \frac{a_{k,l}(K_i^{d,d'})}{c^{\alpha,\beta}_{k} c^{\alpha',\beta'}_{l}} \sum_{i=1}^{\delta(k,d)}\sum_{j=1}^{\delta(l,d')} {\left| \sum_{\mu=1}^n c_\mu S^{d}_{k,i}(x_\mu)\overline{S^{d'}_{l,j}(w_\nu)}\right|}^2=0.
$$
In particular, $c^tAc=0$ if, and only if,
$$
\sum_{\mu=1}^n c_\mu S^{d}_{k,i}(x_\mu)\overline{S^{d'}_{l,j}(w_\nu)}=0,\ \, (k,l)\in J_K,\ \, i\in\{1,2,\hdots,\delta(k,d)\}, \ \, j\in\{1,2,\hdots,\delta(l,d')\}.$$
Reintroducing the addition formula, now leaving a free variable $(x,w) \in \mathbb{M}^d\times \mathbb{H}^{d'}$, the previous assertion implies that
$$
\sum_{\mu=1}^n c_\mu  P_k^{(d-2)/2,\beta}(\cos{(|x_\mu x|/2)})P_l^{(d'-2)/2,\beta'}(\cos{(|w_\mu w|/2)})=0,
$$
for $(x,w)\in\mathbb{M}^d\times \mathbb{H}^{d'}$ and $(k,l)\in J_K$.\
However, if this last assertion holds, it is promptly seen that
$$
\sum_{\mu=1}^n c_\mu \sum_{i=1}^{\delta(k,d)}\sum_{j=1}^{\delta(l,d')} S^{d}_{k,i}(x_\mu)\overline{S^{d}_{k,i}(x)}S^{d'}_{l,j}(w_\mu)\overline{S^{d'}_{l,j}(w)}=0, \quad
 (x,w)\in\mathbb{M}^d\times \mathbb{H}^{d'}, \quad (k,l)\in J_K.
 $$
Using the fact that $\{S^{d'}_{l,1},S^{d'}_{l,2},\hdots,S^{d'}_{l,\delta(l,d')}\}$ and $\{S^{d}_{k,1},S^{d}_{k,2},\hdots,S^{d}_{k,\delta(k,d)}\}$ are basis of $\mathcal{H}^{d'}_{l}$
and $\mathcal{H}^{d}_{k}$, respectively, we are reduced to
$$
\sum_{\mu=1}^n c_\mu S^{d}_{k,i}(x_\mu)\overline{S^{d'}_{l,j}(w_\nu)}=0,\ \, (k,l)\in J_K,\ \, i\in\{1,2,\hdots,\delta(k,d)\}, \ \, j\in\{1,2,\hdots,\delta(l,d')\}.
$$
once again.

The discussion above justifies the following result.

\begin{prop}\label{eqdcspd} Let $K$ be a real, continuous, isotropic and positive definite kernel on $\mathbb{M}^d\times \mathbb{H}^{d'}$.\ The following assertions are equivalent:\\
$(i)$ $K$ is $DC$-strictly positive definite;\\
$(ii)$ If $n\geq 1$, $x_1, x_2, \ldots, x_n$ are distinct points on $\mathbb{M}^d$ and $w_1, w_2, \ldots, w_n$ are distinct points on $\mathbb{H}^{d'}$, then the only solution of the system
$$\left\{\begin{array}{lll}
\sum_{\mu=1}^n  c_\mu P^{(d-2)/2,\beta}_{k}(\cos{(|x_\mu x|/2)}) P^{(d'-2)/2,\beta'}_{l}(\cos{(|w_\mu w|/2)})=0,\\ (x,w) \in \mathbb{M}^d \times \mathbb{H}^{d'},\\ (k,l) \in J_K,
\end{array}\right.$$
is the trivial one, that is, $c_\mu=0$, $\mu=1,2,\ldots, n$.
\end{prop}

In the lemma below, we use the symbol $M_1 \hookrightarrow M_2$ to indicate the existence of an isometric embedding of a metric space $M_1$ into a metric space $M_2$.\ The result is a classical result in the theory of compact two-point homogeneous spaces (see \cite{askey}).

\begin{lem}\label{embed} There exists a chain of isometric embeddings as follows
\begin{equation*}
S^1\hookrightarrow \mathbb{P}^{2}(\mathbb{R}) \hookrightarrow \mathbb{P}^{d}(\mathbb{R})\hookrightarrow \mathbb{P}^{2d}(\mathbb{C})\hookrightarrow \mathbb{P}^{4d}(\mathbb{H})\hookrightarrow \mathbb{P}^{8d}(Cay), \quad d=2,3,\hdots.
\end{equation*}
\end{lem}

In particular, since $\mathbb{P}^2(\mathbb{R})$ is isometrically isomorphic to $S^2$, if the compact two-point homogeneous space $\mathbb{H}^{d'}$ is not a sphere, the lemma guarantees the existence of an integer $q\geq 2$ so that $S^q \hookrightarrow \mathbb{H}^{d'}$.\ On the other hand, this embedding justifies a decomposition of the form
$$
P_l^{(d'-2)/2,\beta'}(s)=\sum_{j=0}^l b_j^l P_{l-j}^{(q-2)/2,(q-2)/2}(s), \quad l=0,1,\ldots,
$$
with all coefficients $b_j^l$ positive.

We will make use of the following normalized Jacobi polynomials
$$R_{k}^{\alpha,\beta}=\frac{P_{k}^{\alpha,\beta}}{P_{k}^{\alpha,\beta}(1)}$$
and some of its properties listed in the lemma below (see \cite{barbosa2,szego}).

\begin{lem}\label{limit} The Jacobi polynomials have the following properties:\\
$(i)$ $P_k^{\alpha,\beta}(-t)=(-1)^k P_k^{\beta,\alpha}(t)$, $t \in [-1,1]$;\\
$(ii)$ $\lim_{k\to \infty} R_{k}^{\alpha,\beta}(t)=0$, $t \in (-1,1)$;\\
$(iii)$ If $\alpha>\beta$, then $\lim_{k\to \infty} P_k^{\beta,\alpha}(1)[P_k^{\alpha,\beta}(1)]^{-1}=0$.
\end{lem}

This is the first characterization for DC-strict positive definiteness we have found.

\begin{thm}\label{dcsuf} Let $K$ be a real, continuous, isotropic and positive definite kernel on $S^d \times \mathbb{H}^{d'}$.\ Assume $d\geq 2$ and that $\mathbb{H}^{d'}$ is not a sphere.\ In order that $K$ be $DC$-strictly positive definite it is necessary and sufficient that  either $\{l: (k,l) \in J_K\ \mbox{for some k}\}$ be infinite or $J_K$ contain two sequences $\{(k_r,l)\}$ and $\{(k_s, l')\}$ for which $\{k_r+l\} \subset 2\mathbb{Z}_+$, $\{k_s+l' \} \subset 2\mathbb{Z}_++1$, and $\lim_{r \to \infty}k_r=\lim_{s \to \infty} k_s=\infty$.
\end{thm}
\begin{proof} Assume $K$ is $DC$-strictly positive definite.\ Recalling Lemma \ref{embed}, it is easily seen that $K$ is $DC$-strictly positive definite on $S^d\times S^q$ for some $q \geq 2$.\ Introducing the equalities presented right after Lemma \ref{embed} into the series representation for
$K_i^{d,d'}$ and arranging leads to
$$
K_i^{d,d'}(t,s)=\sum_{k,l=0}^\infty \left( \sum_{j=0}^\infty a_{k,l+j}(K_i^{d,d'}) b_l^{l+j}  \right)P_k^{(d-2)/2,(d-2)/2}(t)  P_l^{(q-2)/2,(q-2)/2}(s),\ \ t,s \in [-1,1].
$$
In particular, the set
$$\left\{k+l:  \sum_{j=0}^\infty a_{k,l+j}(K_i^{d,d'}) b_l^{l+j}>0\right\}=\left\{k+l:  \sum_{j=0}^\infty a_{k,l+j}(K_i^{d,d'})>0\right\}$$
contains infinitely many even and infinitely many odd integers.\ However, it is not hard to see that if the above condition holds and $\{l: a_{k,l}(K_i^{d,d'})>0\ \mbox{ for some } k\}$ is finite, then $J_K$ must contain two sequences $\{(k_r,l)\}$ and $\{(k_s, l')\}$ for which $\{k_r+l\} \subset 2\mathbb{Z}_+$, $\{k_s+l' \} \subset 2\mathbb{Z}_++1$, and $\lim_{r \to \infty}k_r=\lim_{s \to \infty} k_s=\infty$.\ Indeed, the inferring of this fact demands to observe that if $a_{k,l}(K_i^{d,d'})>0$ for some $(k,l)$, then 
$$k+0, k+1,\ldots, k+l\in \left\{k+l:  \sum_{j=0}^\infty a_{k,l+j}(K_i^{d,d'})>0\right\}.$$ This shows the necessity of the condition.\ As for the sufficiency, let $n$ be a positive integer, $x_1,x_2,\hdots,x_n$ distinct points in $\mathbb{M}^d$ and $w_1,w_2,\hdots,w_n$ distinct points in $\mathbb{H}^{d'}$.\ We will show that, under the condition on $J_K$ mentioned in the statement of the theorem, the only solution of the system
$$\left\{\begin{array}{lll}
\sum_{\mu=1}^n  c_\mu P^{(d-2)/2,\beta}_{k}(\cos{(|x_\mu x|/2)}) P^{(d'-2)/2,\beta'}_{l}(\cos{(|w_\mu w|/2)})=0,\\ (x,w) \in \mathbb{M}^d \times \mathbb{H}^{d'},\\ (k,l) \in J_K,
\end{array}\right.$$ is the trivial one.\ In order to achieve that, we will fix $\gamma \in \{1,2,\ldots, n\}$ and will show that $c_\gamma=0$.\ That will be done trough specific choices of points $x\in \mathbb{M}^d$ and $w\in\mathbb{H}^{d'}$ in the equation defining the system.\ We also need to consider the antipodal index sets
$$\Gamma_{x_\gamma}=\{ \mu:\cos{(|x_\mu x_\gamma|/2)}=-1\}\quad \mbox{and} \quad \Gamma_{w_\gamma}=\{ \mu:\cos{(|w_\mu w_\gamma|/2)}=-1\}.$$
Due to the basic assumptions of the theorem, we know that $d-2=2\beta$ and that $\Gamma_{x_\gamma}$ is unitary, say, $\Gamma_{x_\gamma}=\{\delta\}$.\ The Jacobi polynomials $P^{(d-2)/2,\beta}_{k}$ are then Gegenbauer polynomials and, in particular, they are even functions when $k$ is even and odd functions otherwise.\ The equation defining the system, with the choice $x=x_\delta$ and $w=w_\delta$, can be put into the form
\begin{eqnarray*}
c_\gamma & + &  (-1)^{k+l} \frac{P_{l}^{\beta',(d'-2)/2}(1)}{P_{l}^{(d'-2)/2,\beta'}(1)} \sum_{\mu\in\{\delta\}\cap \Gamma_{w_\gamma}}c_\mu\\
& & +(-1)^{k} \sum_{\mu\in\{\delta\}\setminus \Gamma_{w_\gamma}} c_\mu R_{l}^{(d'-2)/2,\beta'}(\cos{(|w_\mu w_\gamma|/2)})\\
& & \hspace*{1cm} +(-1)^{l} \frac{P_{l}^{\beta',(d'-2)/2}(1)}{P_{l}^{(d'-2)/2,\beta'}(1)}\sum_{\mu\in\Gamma_{w_\gamma}\setminus \{\delta\}} c_\lambda R_{k}^{(d-2)/2,\beta}(\cos{(|x_\mu x_\gamma|/2)})\\
& & \hspace*{1cm}+
\sum_{\mu\notin\{\delta\}\cup \Gamma_{w_\gamma}}c_\mu R_{k}^{(d-2)/2,\beta}(\cos{(|x_\mu x_\gamma|/2)})R_{l}^{(d'-2)/2,\beta'}(\cos{(|w_\mu w_\gamma|/2)})=0.
\end{eqnarray*}
Obviously, some of the sets appearing in the sum decomposition above may be empty.\ Also, the first two sums cannot co-exist, that is, just one of them can appear in the expression.\ If $J_K$ contains a sequence $(k_r, l_r)$ for which $\lim_{r \to \infty} l_r=\infty$, we may conclude that
$$\lim_{r \to \infty} \frac{P_{l_r}^{\beta',(d'-2)/2}(1)}{P_{l_r}^{(d'-2)/2,\beta}(1)}=0,$$
due to Lemma \ref{limit}-$(iii)$,
while
$$\lim_{r \to \infty} R_{l_r}^{(d'-2)/2,\beta'}(\cos{(|w_\mu w_\gamma|/2)})=0, \quad \mu \not \in \Gamma_{x_\gamma},$$
due to Lemma \ref{limit}-$(ii)$.\ Hence, the limit of each summand, but the first, in the previous expression vanishes.\ In particular, $c_\gamma=0$.\  We now proceed assuming the existence of two sequences $\{(k_r,l)\}$ and $\{(k_s, l')\}$ in $J_K$ for which $\{k_r+l\} \subset 2\mathbb{Z}_+$, $\{k_s+l' \} \subset 2\mathbb{Z}_++1$, and $\lim_{r \to \infty}k_r=\lim_{s \to \infty} k_s=\infty$.\ If the second summand in the expression occurs, we can employ these two sequences to deduce that
$$
c_\gamma+\frac{P_{l}^{\beta',(d'-2)/2}(1)}{P_{l}^{(d'-2)/2,\beta'}(1)}c_\delta=c_\gamma-\frac{P_{l'}^{\beta',(d'-2)/2}(1)}{P_{l'}^{(d'-2)/2,\beta'}(1)}
c_\delta =0,
$$
after letting $r \to \infty$ and $s\to \infty$.\ We observe that the limits of the two last summands in the original equation are equal to 0 in this case.\ Now, if $c_\delta \neq 0$, the first equality above provides a contradiction with the positivity of the gamma function in $(0,\infty)$.\ Thus, $0=c_\delta=c_\gamma$.\ Finally, if the third summand is the one occurring in the original expression,
we need an additional equation provided by a second choice of points in the equation defining the system.\ Choosing $x=x_\delta$ and $w=w_\delta$ leads to
\begin{eqnarray*}
c_\delta & + & (-1)^{k+l} \frac{P_{l}^{\beta',(d'-2)/2}(1)}{P_{l}^{(d'-2)/2,\beta'}(1)}  c_\gamma\\
& + & (-1)^k \sum_{\mu\in\{\gamma\}\setminus \Gamma_{w_\delta}}c_\gamma R_l^{(d'-2)/2}(\cos{(|w_\mu w_\delta|/2)}\\
& & \hspace*{1cm} + (-1)^l \frac{P_{l}^{\beta',(d'-2)/2}(1)}{P_{l}^{(d'-2)/2,\beta'}(1)} \sum_{\mu\in \Gamma_{w_\delta}\setminus\{\gamma\}}R_{k}^{(d-2)/2,\beta}(\cos{(|x_\mu x_\delta|/2)})\\
& & \hspace*{1cm} \hspace*{1cm}+
\sum_{\mu\notin\Gamma_{w_\delta}\cup \{\gamma\}}c_\mu R_{k}^{(d-2)/2,\beta}(\cos{(|x_\mu x_\delta|/2)})R_{l}^{(d'-2)/2,\beta'}(\cos{(|w_\mu w_\delta|/2)})=0.
\end{eqnarray*}
Using just one of the sequences, say, $\{(k_r,l)\}$, and letting $r\to \infty$ in both equations, we deduce that
$$
c_\gamma+ \sum_{\mu\in\{\delta\}\setminus \Gamma_{w_\gamma}}c_\mu R_l^{(d'-2)/2}(\cos{(|w_\mu w_\gamma|/2)} = c_\delta + \sum_{\mu\in\{\gamma\}\setminus \Gamma_{w_\delta}}c_\mu  R_l^{(d'-2)/2}(\cos{(|w_\mu w_\delta|/2)} = 0.
$$
But that corresponds to
$$
c_\delta\left[ 1-  \left( R_l^{(d'-2)/2}(\cos{(|w_\gamma w_\delta|/2)}) \right)^2\right] = 0,
$$
with $\cos{(|w_\gamma w_\delta|/2)}\neq \pm 1$.\ Thus, $c_\delta=0$, and consequently, $c_\gamma=0$.
\end{proof}

The next theorem takes care of the the remaining cases.

\begin{thm}\label{dcsufnec} Let $K$ be a real, continuous, isotropic and positive definite kernel on $\mathbb{M}^d\times \mathbb{H}^{d'}$.\ Assume neither $\mathbb{M}^d$ nor $\mathbb{H}^{d'}$ is a sphere.\ In order that $K$ be $DC$-strictly positive definite it is necessary and sufficient that $\{k+l :a_{k,l}(K_i^{d,d'})>0\}$ be infinite.
\end{thm}
\begin{proof} Since the proof is similar to the proof of the previous theorem, some details will be omitted.\ The necessity part is similar to that in the proof of Theorem \ref{dcsuf}, using the same trick twice.\ The resulting kernel is $DC$-strictly positive definite on some $S^q \times S^q$ and the set of indices pertaining to the final argument takes the form
$$\left\{k+l:  \sum_{j=0}^\infty \sum_{j'=0}^\infty a_{k+j,l+j'}(K_i^{d,d'}) b_k^{k+j}c_l^{l+j'}>0\right\}=\left\{k+l:  \sum_{j=0}^\infty \sum_{j'=0}^\infty a_{k+j,l+j'}(K_i^{d,d'})>0\right\},$$
where all the constants $b_k^{k+j}c_l^{l+j'}$ are positive.\ Since this set has infinitely many even and infinitely many odd integers, it follows that $\{k+l: (k,l) \in J_K\}$ is infinite.\ The sufficiency part follows the steps of the corresponding part in the previous theorem.\ Due to the assumption on $J_K$, we can select a sequence $\{(k_r,l_r)\}$ in $J_K$ so that either $\lim_{r \to \infty} k_r=\infty$ or $\lim_{r \to \infty} l_r=\infty$.\
Choosing $x=x_\gamma$ and $w=w_\gamma$ in the equation defining the system, we obtain
\begin{eqnarray*}
c_\gamma & + &  (-1)^{k_r+l_r} \frac{P_{k_r}^{\beta,(d-2)/2}(1)}{P_{k_r}^{(d-2)/2,\beta}(1)}\frac{P_{l_r}^{\beta',(d'-2)/2}(1)}{P_{l_r}^{(d'-2)/2,\beta'}(1)} \sum_{\mu \in \Gamma_{x_\gamma} \cap \Gamma_{w_\gamma}} c_\mu\\
& & +(-1)^{k_r} \frac{P_{k_r}^{\beta,(d-2)/2}(1)}{P_{k_r}^{(d-2)/2,\beta}(1)}\sum_{\mu\in\Gamma_{x_\gamma}\setminus \Gamma_{w_\gamma}} c_\mu R_{l_r}^{(d'-2)/2,\beta'}(\cos{(|w_\mu w_\gamma|/2)})\\
& & \hspace*{1cm} +(-1)^{l_r} \frac{P_{l_r}^{\beta',(d'-2)/2}(1)}{P_{l_r}^{(d'-2)/2,\beta'}(1)}\sum_{\mu\in\Gamma_{w_\gamma}\setminus \Gamma_{x_\gamma}} c_\lambda R_{k_r}^{(d-2)/2,\beta}(\cos{(|x_\mu x_\gamma|/2)})\\
& & \hspace*{1cm}+
\sum_{\mu\notin\Gamma_{x_\gamma}\cup \Gamma_{w_\gamma}}c_\mu R_{k_r}^{(d-2)/2,\beta}(\cos{(|x_\mu x_\gamma|/2)})R_{l_r}^{(d'-2)/2,\beta'}(\cos{(|w_\mu w_\gamma|/2)})=0,
\end{eqnarray*}
If $\lim_{r \to \infty} k_r=\infty$, then the limit of each summand, but the first, vanishes.\ In particular, $c_\gamma=0$.\ If $\lim_{r \to \infty} l_r=\infty$, a similar analysis produces the same conclusion.
\end{proof}

\section{Strict positive definiteness}

The strict positive definiteness of a real, continuous, isotropic and positive definite kernel on a product of high dimensional spheres was completely characterized in \cite{guella3} while the characterization in the case of a product of circles was reached in \cite{guella2}.\ Thus, just like in the previous section, we will assume that at least one of the spaces involved is not a sphere.

The section begins with the obvious counterpart of Proposition \ref{eqdcspd} for plain strict positive definiteness on $\mathbb{M}^d\times \mathbb{H}^{d'}$.

\begin{prop}\label{eqdcspd} Let $K$ be a real, continuous, isotropic and positive definite kernel on $\mathbb{M}^d\times \mathbb{H}^{d'}$.\ The following assertions are equivalent:\\
$(i)$ $K$ is strictly positive definite;\\
$(ii)$ If $n\geq 1$ and $(x_1,w_1), (x_2,w_2), \ldots, (x_n,w_n)$ are distinct points on $\mathbb{M}^d \times \mathbb{H}^{d'}$, then the only solution of the system
$$\left\{\begin{array}{lll}
\sum_{\mu=1}^n  c_\mu P^{(d-2)/2,\beta}_{k}(\cos{(|x_\mu x|/2)}) P^{(d'-2)/2,\beta'}_{l}(\cos{(|w_\mu w|/2)})=0,\\ (x,w) \in \mathbb{M}^d \times \mathbb{H}^{d'},\\ (k,l) \in J_K,
\end{array}\right.$$
is the trivial one, that is, $c_\mu=0$, $\mu=1,2,\ldots, n$.
\end{prop}

The characterization for strict positive definiteness in the case in which both spaces are not spheres is as follows.

\begin{thm}\label{spdnec} Let $K$ be a real, continuous, isotropic and positive definite kernel on $\mathbb{M}^d\times \mathbb{H}^{d'}$.\ Assume that neither $\mathbb{M}^d$ nor $\mathbb{H}^{d'}$ is a sphere.\ In order that $K$ be strictly positive definite it is necessary and sufficient that the set $J_K$ contains a sequence $\{(k_r,l_r)\}$ for which $\lim_{r\to\infty}k_r=\lim_{r\to\infty} l_r =\infty$.
\end{thm}
\begin{proof} Assume there exists a sequence $\{(k_r,l_r)\}$ as described in the statement of the theorem.\ Let $(x_1,w_1),(x_2,w_2),\hdots,(x_n,w_n)$ be distinct points in $\mathbb{M}^d\times \mathbb{H}^{d'}$ and suppose that
$$
\sum_{\mu=1}^{n}c_\mu P^{(d-2)/2,\beta}_{k}(\cos{(|x_\mu w|/2)}) P^{(d'-2)/2,\beta'}_{l}(\cos{(|w_\mu w|/2)}=0,
$$
for real scalars $c_1,c_2, \ldots, c_n$, $(x,w)\in \mathbb{M}^d\times \mathbb{H}^{d'}$ and $(k,l) \in J_K$.\ For $\gamma\in\{1,2,\hdots,n\}$ fixed, let us put $x=x_\gamma$ and $w=w_\gamma$ in the previous equation and split it taking into account the following index sets (recall the normalization we have adopted for the metric in the spaces involved):
$$I_1=\{\mu: |x_\mu x_\gamma|=2\pi=|w_\mu w_\gamma|\},$$
$$I_2=\{\mu :|x_\mu x_\gamma|=2\pi\neq |w_\mu w_\gamma|\},$$
and
$$I_3=\{\mu : |x_\mu x_\gamma|\neq 2\pi = |w_\mu w_\gamma|\}.$$
We observe that one or more of these sets may be empty.\ The outcome is
\begin{eqnarray*}
c_\gamma P^{(d-2)/2,\beta}_{k}(1)P^{(d'-2)/2,\beta'}_{l}(1) & + & (-1)^{k+l} P_k^{\beta,(d-2)/2}(1)P_l^{\beta',(d'-2)/2}(1) \sum_{\mu \in I_1} c_\mu \\
& + &   (-1)^{k} P_k^{\beta,(d-2)/2}(1)\sum_{\mu\in I_2} c_\mu P_l^{(d'-2)/2,\beta'}(\cos{(|w_\mu w_\gamma|/2)})\\
 & & \hspace{-15mm}+ \ (-1)^{l} P_l^{\beta',(d'-2)/2}(1)\sum_{\mu\in I_3} c_\mu P_k^{(d-2)/2,\beta}(\cos{(|x_\mu x_\gamma|/2)})  \\
& & \hspace{-30mm} + \ \sum_{\mu \notin I_1 \cup I_2 \cup I_3}c_\mu P_k^{(d-2)/2,\beta}(\cos{(|x_\mu x_\gamma|/2)})P_l^{(d'-2)/2,\beta'}(\cos{(|w_\mu w_\gamma|/2)}) = 0.
\end{eqnarray*}
A small adjustment implies that
$$\hspace*{-60mm} c_\gamma +(-1)^{k_r+l_r} \frac{P_{k_r}^{\beta,(d-2)/2}(1)}{P_{k_r}^{(d-2)/2,\beta}(1)}\frac{P_{l_r}^{\beta',(d'-2)/2}(1)}{P_{l_r}^{(d'-2)/2,\beta'}(1)} \sum_{\mu\in I_1} c_\mu$$
$$ \hspace*{-30mm} + (-1)^{k_r} \frac{P_{k_r}^{\beta,(d-2)/2}(1)}{P_{k_r}^{(d-2)/2,\beta}(1)}\sum_{\mu\in I_2} c_\mu R_{l_r}^{(d'-2)/2,\beta'}(\cos{(|w_\mu w_\gamma|/2)})$$
$$+ (-1)^{l_r} \frac{P_{l_r}^{\beta',(d'-2)/2}(1)}{P_{l_r}^{(d'-2)/2,\beta'}(1)}\sum_{\mu \in I_3} c_\mu R_{k_r}^{(d-2)/2,\beta}(\cos{(|x_\mu x_\gamma|/2)})$$
$$ \hspace*{5mm}+  \sum_{\mu\notin I_1 \cup I_2 \cup I_3}c_\mu R_{k_r}^{(d-2)/2,\beta}(\cos{(|x_\mu x_\gamma|/2)})R_{l_r}^{(d'-2)/2,\beta'}(\cos{(|w_\mu w_\gamma|/2)}) =0, \quad r=1,2,\ldots.$$
We observe that in the last summand, if $\mu$ is fixed, either $x_\gamma \neq x_\mu$ or $w_\gamma \neq w_\mu$.\ Since $\alpha > \beta$ and $\alpha' > \beta'$, we may let $r\to \infty$ and apply Lemma \ref{limit} to conclude that $c_\gamma=0$.\ In view of the previous proposition, the sufficiency part is resolved.\
Going the other way around, if $K$ is strictly positive definite, we may repeat the procedure adopted in the first half of the proof of Theorem \ref{dcsufnec}.\ The index set of the resulting positive definite kernel on $S^q \times S^q$ is
$$\left\{(k,l):  \sum_{j=0}^\infty \sum_{j'=0}^\infty a_{k+j,l+j'}(K_i^{d,d'})>0\right\}.$$
Since the characterization for strict positive definiteness on $S^q \times S^q$ described in \cite{guella3} implies that the set above must contain at least one sequence $(k_r,l_r)$ for which $\lim_{r \to \infty}k_r=\lim_{r\to \infty} l_r=\infty$, the set $J_K$ must contain a sequence of this same type.\end{proof}

In the case in which $\mathbb{M}^{d}=S^{d}$, the following upgrade of the previous lemma will be more favorable.

\begin{lem} \label{matrix2} Let $K$ be a real, continuous, isotropic and positive definite kernel on $S^d \times \mathbb{H}^{d'}$, in which $d\geq 2$ and $\mathbb{H}^{d'}$ is not a sphere.\ The following statements are equivalent:\\
$(i)$ $K$ is strictly positive definite on $S^d \times \mathbb{H}^{d'}$;\\
$(ii)$ If $n\geq 1$, $(x_1,w_1),(x_2,w_2)\hdots,(x_n,w_n)$ are distinct points on $S^d \times \mathbb{H}^{d'}$, and the set $\{x_1,x_2, \ldots, x_n\}$ does not contain any pair of antipodal points, then the only solution of the system
$$\left\{\begin{array}{lll}
\sum_{\mu=1}^n  \left[(-1)^k c_\mu'+c_\mu'' \right]P^{(d-2)/2,\beta}_{k}(x_\mu \cdot x) P^{(d'-2)/2,\beta'}_{l}(\cos{(|w_\mu w|/2)})=0,\\
(x,w) \in S^d \times \mathbb{H}^{d'}\\
(k,l) \in J_K,
\end{array}\right.
$$
is $c_\mu'=c_\mu''=0$, $\mu=1,2,\hdots,n$.
\end{lem}
\begin{proof} Assume $(i)$ holds.\ Let $(x_1,w_1),(x_2,w_2),\hdots,(x_n,w_n)$ be distinct points in $S^d \times \mathbb{H}^{d'}$ and assume that $\{x_1,x_2, \ldots, x_n\}$ does not contain pairs of antipodal points.\
Since $2\beta=d-2$, the system described in $(ii)$ can be written in the form
$$
\sum_{\nu=1}^{2n}  c_\nu P^{(d-2)/2,\beta}_{k}(\cos{(|x_\nu' x|/2)}) P^{(d'-2)/2,\alpha'}_{l}(\cos{(|w_\nu' w|/2)})=0, \quad x\in S^d,\quad w\in \mathbb{H}^{d'},
$$
in which $(x_\nu',w_\nu')=(x_\nu,w_\nu)$ and $c_\nu=c_\nu'$ if $\nu\in \{1,2,\hdots,n\}$ and $(x_\nu',w_\nu')=(-x_\nu,w_\nu)$ and $c_\nu=c_\nu''$ if $\nu\in\{n+1,2,\hdots,2n\}$.\ Since the $2n$ points $(x_\nu',w_\nu')$ are distinct, Proposition \ref{eqdcspd} implies that $c_\nu=0$, $\nu=1,2,\hdots,2n$.\ In particular, $c_\mu'=c_\mu''=0$, $\mu=1,2,\hdots,n$.\\
Conversely, if $(i)$ does not hold, the previous proposition allows the selection of distinct point $(x_1,w_1),(x_2,w_2),\hdots,(x_n,w_n)$ in $S^d \times \mathbb{H}^d$ so that the system
$$\left\{\begin{array}{lll}
\sum_{\mu=1}^{n}  c_\mu P^{(d-2)/2,\beta}_{k}(\cos{(|x_\mu x|/2)}) P^{(d'-2)/2,\alpha'}_{l}(\cos{(|w_\mu w|/2)})=0,\\
(x,w) \in S^d \times \mathbb{H}^{d'},\\
(k,l) \in J_K,
\end{array}\right.
$$
has a nontrivial solution $c_\mu$, $\mu=1,2,\ldots,c_n$.\ We can select $p$ $(\leq n)$ distinct points $\{(x_1',w_1'),(x_2',w_2'),\hdots,(x_p',w_p')\}$ in $S^d \times \mathbb{H}^{d'}$ in a such a way that $\{x_1', x_2',\ldots, x_p'\}$ contains no pairs of antipodal points and
$$
\{(x_1,w_1),(x_2,w_2),\hdots,(x_n,w_n)\subseteq \{(\pm x_1',w_1'),(\pm x_2',w_2'),\hdots,(\pm x_p',w_p')\}.
$$
However, it is an easy matter to verify that the system
\begin{equation*}
\left\{\begin{array}{lll}
\displaystyle \sum_{\mu=1}^n  \left[(-1)^kc_\mu'+ c_\mu'' \right]P^{(d-2)/2,\beta}_{k}(\cos{(|x_\mu' x|/2)}) P^{(d'-2)/2,\alpha'}_{l}(\cos{(|w_\mu' w|/2)})=0,\\
(x, w) \in S^d \times \mathbb{H}^{d'},\\
(k,l) \in J_K,
\end{array}\right.
\end{equation*}
has a nontrivial solution as well.\ Thus, $(ii)$ cannot hold.
\end{proof}

The following proposition is an alternative to the previous lemma via the sets
$$J_K^{e}:=J_K\cap [2\mathbb{Z}_+\times \mathbb{Z}_+]  \quad \mbox{and} \quad J_K^{o}:=J_K\cap [(2\mathbb{Z}_+ +1)\times  \mathbb{Z}_+].$$

\begin{prop} \label{matrix5} Let $K$ be a real, continuous, isotropic and positive definite kernel on $S^d\times \mathbb{H}^{d'}$, in which $d\geq 2$ and $\mathbb{H}^{d'}$ is not a sphere.\ The following statements are equivalent:\\
$(i)$ $K$ is strictly positive definite on $S^d\times \mathbb{H}^{d'}$;\\
$(ii)$ If $n\geq 1$, $(x_1,w_1),(x_2,w_2),\hdots,(x_n,w_n)$ are distinct points on $S^d\times \mathbb{H}^{d'}$, and the set $\{x_1,x_2, \ldots, x_n\}$ does not contain a pair of antipodal points, then the only solution of the system
\begin{equation*}
\left\{\begin{array}{lll}
\displaystyle \sum_{\mu=1}^n  c_\mu^e P^{(d-2)/2,\beta}_{k}(\cos{(|x_\mu x|/2)}) P^{(d'-2)/2,\beta'}_{l}(\cos{(|w_\mu w|/2)})=0,\quad  (k,l)\in J_{K}^{e},\\
\displaystyle \sum_{\mu=1}^n  c_\mu^o P^{(d-2)/2,\beta}_{k'}(\cos{(|x_\mu x|/2)}) P^{(d'-2)/2,\beta'}_{l'}(\cos{(|w_\mu w|/2)})=0,\quad (k',l')\in J_{K}^{o}, \\
(x,w)\in S^d\times \mathbb{H}^{d'},
\end{array}\right.
\end{equation*}
 is $c_\mu^e=c_\mu^o=0$, $\mu=1,2,\hdots,n$.
\end{prop}

\begin{proof} If $(ii)$ were not true, we could find distinct points $(x_1,w_1),(x_2,w_2),\hdots,$\linebreak $(x_n,w_n)$ in $S^d\times \mathbb{H}^{d'}$, with $\{x_1,x_2, \ldots, x_n\}$ containing no pair of antipodal points and either
\begin{equation*}
\left\{\begin{array}{ll}
\sum_{\mu=1}^n  c_\mu^e P^{(d-2)/2,\beta}_{k}(\cos{(|x_\mu x|/2)}) P^{(d'-2)/2,\beta'}_{l}(\cos{(|w_\mu w|/2)})=0,\quad  (k,l)\in J_{K}^{e},\\
(x,w)\in S^d\times \mathbb{H}^{d'},\\
\end{array}\right.
\end{equation*}
or
\begin{equation*}
\left\{\begin{array}{ll}
\sum_{\mu=1}^n  c_\mu^o P^{(d-2)/2,\beta}_{k'}(\cos{(|x_\mu x|/2)}) P^{(d'-2)/2,\beta'}_{l'}(\cos{(|w_\mu w|/2)})=0,\quad (k',l')\in J_{K}^{o},\\
(x,w)\in S^d\times \mathbb{H}^{d'}
\end{array}\right.
\end{equation*}
having a nontrivial solution.\ We proceed considering the first possibility that emerges from the conclusion above, being the other case similar.\ For each $\mu \in\{1,2,\hdots,n\}$, the system
$$
\left\{\begin{array}{ccc}
c_\mu'+c_\mu'' & = & c_\mu^e\\
-c_\mu'+c_\mu'' & = & 0
\end{array}, \right.
$$
has a unique solution $c_\mu',c_\mu''$.\ Since $c_\mu^{e}\neq 0$ for at least one $\mu$, then $(c_\mu',c_\mu'') \neq 0$, for at least one $\mu$.\ It is now clear that
the system in Lemma \ref{eqdcspd}-$(ii)$ would have a nontrivial solution for the selection of points $(x_1,w_1),(x_2,w_2),\hdots,(x_n,w_n)$.\ Thus, $(i)$ implies $(ii)$.\\
The converse will be justified as long as we show that if $(ii)$ holds, then Lemma \ref{eqdcspd}-$(ii)$ holds.\ But, if $c_\mu^{e}$ and $c_\mu^{o}$ are known, the system
 $$
\left\{\begin{array}{ccc}
c_\mu'+c_\mu'' & = & c_\mu^e\\
-c_\mu'+c_\mu'' & = & c_\mu^o
\end{array}, \right.
$$ always has a unique solution.\ If $c_\mu^{e}=c_\mu^{o}=0$ for all $\mu$, then the solution of the corresponding system vanishes.\ Thus, if the system in $(ii)$ has the trivial solution only,
 the same will be true of the system in Lemma \ref{eqdcspd}-$(ii)$.
\end{proof}

We are ready to state and prove the last main contribution of the paper.

\begin{thm}\label{spdnec1} Let $K$ be a real, continuous, isotropic and positive definite kernel on $S^d \times \mathbb{H}^{d'}$.\ Assume that $d\geq 2$ and that $\mathbb{H}^{d'}$ is not a sphere.\ In order that $K$ be strictly positive definite it is necessary and sufficient that the set $J_K$ contain sequences $\{(k_r,l_r)\}$ and $\{(k_r',l_r')\}$ so that $\{k_r\}\subset 2\mathbb{Z}_+$, $\{k_r'\} \subset 2\mathbb{Z}_++1$, and  $\lim_{r\to\infty}k_r=\lim_{r\to\infty}k_r'=\lim_{r\to\infty} l_r=\lim_{r\to\infty} l_r' =\infty$.
\end{thm}
\begin{proof}
Let us assume that $J_K$ contains sequences as described in the statement of the theorem.\ We intend to use Proposition \ref{matrix5} in order to conclude that $K$ is strictly positive definite.\ Let $(x_1,w_1),(x_2,w_2),\hdots,(x_n,w_n)$ be distinct points in $S^d \times \mathbb{H}^{d'}$, assume that $\{x_1, x_2, \ldots, x_n\}$ does not contain any pairs of antipodal points and that the system in Proposition \ref{matrix5}-$(ii)$ holds.\
Fixing $\gamma$, introducing $x=x_\gamma$ and $w=w_\gamma$ in the first equation of the system and proceeding as in the proof of Theorem \ref{spdnec}, we deduce that
\begin{eqnarray*}
c_\gamma^e +(-1)^{l} \frac{P_{l}^{\beta',(d'-2)/2}(1)}{P_{l}^{(d'-2)/2, \beta'}(1)} \sum_{\mu\in I_1} c_\mu^e R_{k}^{(d-2)/2,\beta}(\cos{(|x_\mu x_\gamma|/2)})\\
& & \hspace{-80mm} +  \sum_{\mu \in I_2}c_\mu^e R_{k}^{(d-2)/2,\beta}(\cos{(|x_\mu x_\gamma|/2)})R_{l}^{(d'-2)/2,\beta'}(\cos{(|w_\mu w_\gamma|/2)}) =0,
\end{eqnarray*}
in which the index sets are now $I_1=\{\mu: |w_\mu w_\gamma|=2\pi\}$ and $I_2=\{\mu :|w_\mu w_\gamma|\neq 2\pi\}$.
Substituting the first double sequence guaranteed by our assumption in this equation and letting $r \to \infty$, Lemma \ref{limit} implies that $c_\gamma^e=0$.\ A similar procedure with the second equation of the system and with the second double sequence from the assumption leads to $c_\gamma^o=0$.\ Since $\gamma$ is arbitrary, the only solution of the system in Proposition \ref{matrix5}-$(ii)$ is the trivial one.\ Therefore, $K$ is strictly positive definite on $S^d \times \mathbb{H}^{d'}$.\ In order to prove the condition is necessary, we need to imitate the corresponding part in the proof of Theorem \ref{dcsuf}.\ We get a strictly positive definite kernel on $S^d \times S^q$ with corresponding index set
$$\left\{(k,l):  \sum_{j=0}^\infty a_{k,l+j}(K_i^{d,d'})>0\right\}.$$
Once again, the characterization for strict positive definiteness on $S^d \times S^q$ described in \cite{guella3} implies that $J_k$ must contain two sequences as quoted in the statement of the theorem being proved.\end{proof}

As a final remark, we would like to observe that it is still an open problem to obtain versions of the theorems proved in Sections 3 and 4 in the cases in which the spaces are different and at least one of them is $S^1$.\ As a matter of fact, for the theorems proved in Sections 3, it is also open the case in which both spaces are $S^1$.


%
%

\vspace*{2cm}

\noindent V. S. Barbosa and V. A. Menegatto \\
Departamento de
Matem\'atica,\\ ICMC-USP - S\~ao Carlos, Caixa Postal 668,\\
13560-970 S\~ao Carlos SP, Brasil\\ e-mails: victorrsb@gmail.com; menegatt@icmc.usp.br

\end{document}